\documentclass[12pt,reqno]{amsart}
\usepackage{fullpage}

\usepackage{ifpdf}
\ifpdf
   \usepackage{times}
\fi

\usepackage[T1]{fontenc}
\usepackage{mathrsfs}
\usepackage{epsfig}
\usepackage{color}

\newtheorem{theorem}{Theorem} [section]
\newtheorem{lemma}[theorem]{Lemma}

\newtheorem{prop}[theorem]{Proposition}

\theoremstyle{definition}

\newtheorem{remark}[theorem]{Remark}
\numberwithin{equation}{section}

\renewcommand{\mod}[1]{{\ifmmode\text{\rm\ (mod~$#1$)}\else\discretionary{}{}{\hbox{ }}\rm(mod~$#1$)\fi}}

\newsymbol\nmid 232D

\newcommand{\rad}{{\rm rad}}

\newcommand{\PP}{{\mathcal P}}
\newcommand{\Nn}{{\mathcal N}}
\newcommand{\e}{{\rm e}}

\renewcommand{\S}{{\mathcal S}}
\newcommand{\T}{{\mathcal T}}
\newcommand{\Z}{{\mathbb Z}}

\newcommand{\buv}{\beta_{\{u,v\}}}

\newcommand{\starsum}{\sideset{}{^*}{\sum}}

\begin{document}

\title{The maximal density of product-free sets
in $\Z/n\Z$}

\author[P. Kurlberg]{P\"ar Kurlberg}
\address{Department of Mathematics\\
KTH\\
SE-10044, Stockholm, Sweden}
\email{kurlberg@math.kth.se}

\author[J. C. Lagarias] {Jeffrey C.  Lagarias}
\address{Department of Mathematics\\
University of Michigan\\
Ann Arbor, MI 48109, USA}
\email{lagarias@umich.edu}

\author[C. Pomerance]{Carl Pomerance}
\address{Mathematics Department\\
Dartmouth College\\
Hanover, NH 03755, USA}
\email{carl.pomerance@dartmouth.edu}

\subjclass[2000]{11B05, 11B75}

\date{January 10, 2012}
\maketitle

\begin{abstract}
  This paper studies the maximal size of product-free sets in
  $\Z/n\Z$.  These are sets of residues for which there is no solution
  to $ab\equiv c\pmod n$ with $a,b,c$ in the set.  
  In a previous paper we constructed an infinite sequence of integers
  $(n_{i})_{i\geq1}$ and product-free sets $\S_{i}$ in $\Z/n_{i}\Z$ such that
  the density $|\S_{i}|/n_{i}\to1$ as $i\to\infty$,
  where $|\S_{i}|$ denotes the cardinality of $\S_{i}$.
Here we obtain matching, up to constants, upper and lower  bounds on the 
maximal attainable density as $n \to \infty$.

\end{abstract}

%
%
\section{Introduction}

An important problem in combinatorial number
theory is the study of sets of integers with additive restrictions.
For example, a sum-free set $\S$ is one  forbidding solutions to
$a+b=c$  with $a,b, c \in \S$, and 
the condition of requiring  no solutions to $a+c= 2b$ gives sets $\S$ containing
no three-term arithmetic progression.
For sum-free sets it is easy to show that such sets have upper density at most $\frac{1}{2}$,
and the same holds for subsets of $\Z/n\Z$, and more generally for finite
abelian groups.
In fact, by the work of Green and Ruzsa \cite{GR05} (building on
partial results by Diananda and Yap \cite{DY69}),
the density attainable for any finite abelian group is known. 

Similarly, it is also natural to consider sets with multiplicative
restrictions.  For example, Behrend, Besicovitch, Erd\H os and others (see Hall~\cite{H96}) 
considered sets of integers with no member properly dividing another (known as {\em primitive} sets), 
and Erd\H os~\cite{E38} considered sets where no member divides the product of 
two other members.

Here we consider a multiplicative version of the sum-free problem.
We say  a set of integers $\S$ is {\em product-free} if whenever 
$a,b,c\in \S$ we have $ab\ne c$.  Similarly, if $\S\subset\Z/n\Z$, 
we say $\S$ is product-free if $ab\not\equiv c\pmod n$, 
whenever $a,b,c\in\S$.  Clearly, if $\S$ is a product-free
subset of $\Z/n\Z$, then the set of integers congruent modulo~$n$ to 
some member of $\S$ is a product-free set of integers.
For a product-free subset $\S$ of $\Z/n\Z$, let $D(\S)=|\S|/n$, where
$|\S|$ denotes the cardinality of $\S$.  Further,
let $D(n)$ denote the maximum of $D(\S)$ over all product-free sets $\S\subset\Z/n\Z$.

The problem of product-free sets in $\Z/n\Z$ was studied in a recent paper
by the third author and Schinzel \cite{PS11}.
One might initially think that this product-free problem has a similar
answer to the sum-free case, where the density can never exceed 
$\frac{1}{2}$.  In this direction, it was shown in \cite{PS11} that 
$D(n)<\frac{1}{2}$ holds 
for the vast majority of numbers $n$; specifically for all $n$ except possibly those
divisible by some $m^2$ where $m$ is the product of 6 distinct primes, 
and consequently the possible exceptional set has upper density smaller than
$1.56 \times 10^{-8}$. 
However, somewhat surprisingly, there are numbers $n$ for which 
$D(n)$ is arbitrarily close to~1;  in~\cite{KLP11} it was
shown that there are infinitely many $n$ such that 
\begin{equation}
  \label{eq:old-lower-bound}
D(n)>1-\frac{C}{(\log\log n)^{1-\frac12\e\log2}}
\end{equation}
for a suitable positive constant $C$. Here the exponent $1-
\frac{1}{2} \e \log 2 \approx 0.057915$.  Some key features of the sets
$\S$ of high density so constructed are that $n$ is highly composite,
divisible by the square of each of its prime factors,
and each member of such a set has a large common divisor with $n$.

Our aim in this paper is to get an exact form for the rate at which
$D(n)$ can approach $1$.
We begin with an upper bound that closely matches the lower bound
(\ref{eq:old-lower-bound}).
\begin{theorem}
\label{thm12}
There is a positive constant $c$ such that for all $n\ge20$,
$$
 D(n)<1-\frac{c}{(\log\log n)^{1-\frac{1}{2}\e\log2}\sqrt{\log\log\log n}}.
$$
\end{theorem}
\noindent
The restriction to $n\ge20$
is made here so  that the triple logarithm
is defined and positive.
Our second result is an improvement of the lower bound
(\ref{eq:old-lower-bound}) which shows that, up to  constants,
Theorem~\ref{thm12} is sharp.
\begin{theorem}
\label{thm11}
There is a positive constant $C$ and infinitely many integers $n$ with
$$
 D(n)>1-\frac{C}{(\log\log n)^{1-\frac{1}{2}\e\log2} \sqrt{\log\log\log n}}.
$$
\end{theorem}

Before proceeding, we give a brief outline of the proof of
our principal result, Theorem 
\ref{thm12}.  
To bound the maximum density from above, we introduce certain linear
programming (LP) problems $(P_n)$.
The variables of  $(P_n)$
are $\{\alpha_u\}$ with  $u$ ranging over the divisors of $n$ exceeding~1,
with objective function $\sum\alpha_u/u$.
Given a product-free set $\S$, the values
$$\alpha_{u}=|\{ a \in \S : (a,n)=u\}|/|\{ a \mod n : (a,n)=u\}|,$$
for $u>1$ give a feasible solution to $(P_n)$.
There is a mismatch between the objective function 
and $D(\S)$,
and to get around this we  associate to each $n$ a larger auxiliary
number $N= N(n)$ 
which $n$ divides (so that $D(n)\le D(N)$), such that the  optimal
solution value of the linear program $(P_N)$ 
can be used to give an upper bound on $D(N)$
(Theorem~\ref{th41a}).
To bound the new optimal solution value,  we switch to the dual linear program $(D_N)$,
for which each feasible solution gives an upper bound on the
optimal value of $(P_N)$.  A mechanism for finding a good feasible solution
to the dual LP is the heart of the proof given in Section \ref{sec5}.

There remains the problem of  obtaining  tight optimal constants  in
these theorems. 
 With some effort, numerical values for
$c$ and $C$ in Theorems~\ref{thm12} and~\ref{thm11} are computable.
However the  linear program used to prove Theorem~\ref{thm12} relaxes the conditions of the problem
and loses some information, and it is perhaps unlikely that the constants $c$ and $C$ so obtained will
asymptotically match.

The proof of Theorem~\ref{thm12} is given in Sections
\ref{sec2}-\ref{sec5}.  In Section \ref{sec6}, 
we prove Theorem~\ref{thm11} by refining the method of
\cite{KLP11}.\medskip

\paragraph{\bf Notation.} For  $n$ a positive integer, $\varphi(n)=|(\Z/n\Z)^{*}|$ denotes Euler's  function at $n$,
$\omega(n)$ denotes the number of distinct prime factors of $n$, $\Omega(n)$ denotes the
total number of prime factors of $n$ counted with multiplicity, $\sigma(n)$ denotes
the sum of the positive divisors of $n$, and $\rad(n)$ denotes the largest squarefree divisor
of $n$.  We write $d\| n$ if $d\mid n$ and $\gcd(d,n/d)=1$.
We use the notation $A(x)\ll B(x)$ if $A(x)=O(B(x))$.  This relation is uniform in other variables
unless indicated by a subscript.  We write $A(x)\asymp B(x)$ if $A(x)\ll B(x)\ll A(x)$.  Finally,
we always use the letter $p$ to denote a prime variable.

%
%
\section{Preliminaries: Properties of the Density Function}\label{sec2}

As noted in \cite{KLP11}, we have the following simple
result.


\begin{lemma}\label{lem21}
For all integers $m,n \ge 1$,
\begin{equation}\label{211}
D(n) \le D(mn).
\end{equation}
\end{lemma}

\begin{proof}
Given a product-free set $\S~ (\bmod~n)$, the set 
$\tilde{\S} := \S +\{ 0, n, 2n,..., (m-1)n\} \subset \Z/mn\Z$ 
has $|\tilde{\S} |= m|\tilde{\S}|$.
Now $\tilde{\S}$ is product-free $(\bmod ~mn)$ since any product of elements in $\tilde{\S}$
falls
in a congruence class $(\bmod ~n)$ that is not in $\S$.
\end{proof}

For a positive integer $n$ and a divisor $u$ of $n$, we let
$$\T_u:= \{a \in \Z/n\Z: ~~\gcd (a, n)= u\}.$$
Clearly
\begin{equation}\label{eq225}
|\T_u| = \varphi\left(\frac{n}{u}\right).
\end{equation}
Given some subset $\S$ of $\Z/n\Z$, we let
$$
\S_u := \{ a \in \S: ~\gcd(a, n)=u\}=\S\cap\T_u. 
$$
It is natural to measure the size of $\S_u$ with respect to $\T_u$.

The following result is implicit in \cite{PS11};
since it is central to our argument, we give complete details.

\begin{lemma}~\label{lem23} 
For any product-free set $\S$ $(\bmod ~n)$ and $u\mid n$, let 
$$
\alpha_u=\alpha_u(\S):=\frac{|\S_u|}{|\T_u|}=\frac{|\S_u|}{\varphi(n/u)}.
$$
Then, for all $v\mid n$ such that $uv\mid n$, we have
\begin{equation}\label{eq231}
0 \le \alpha_u \le 1
\end{equation}
and
\begin{equation}\label{eq232}
\alpha_u + \alpha_v + \alpha_{uv} \le 2
\end{equation}
\end{lemma}

\begin{proof}
Here \eqref{eq231} is immediate, holding in fact
for any set $\S \subset \Z/n\Z$,
whether or not it is product-free. 
If $\alpha_u=0$, then \eqref{eq232} immediately follows from
\eqref{eq231} applied to $v$ and $uv$, so we may assume that
$\alpha_u>0$.  Let $a\in\S_u$.  In the ring $\Z/n\Z$, multiplication
by $a$ takes $\T_v$ onto $\T_{uv}$, where each member of $\T_{uv}$
has the same size pre-image in $\T_v$, namely 
$|\T_v|/|\T_{uv}|= \varphi(n/v)/\varphi(n/uv)=k$,
say.  Since $\S$ is product-free, each $b\in\S_{uv}$ is thus associated
with $k$ members of $\T_v$ that cannot lie in $\S_v$.  Thus,
$k|\S_{uv}|+|\S_v|\le|\T_v|=\varphi(n/v)$.  Dividing this inequality
by $\varphi(n/v)$ and using the definition of $k$ gives
$$
\frac{|\S_{uv}|}{\varphi(n/uv)}+\frac{|\S_v|}{\varphi(n/v)}\le1,
$$
which with \eqref{eq231} proves \eqref{eq232}.
\end{proof}

Finally we recall (from~\cite{PS11}) a fact about product-free sets $\S$.

\begin{lemma}\label{lem22}
 Given $n$, if $\S$ is product-free $(\bmod ~n)$ and $a \in \S$
has $\gcd(a, n) =1$, then
$$
D(\S) < \frac{1}{2}.
$$
Thus, if $D(\S)\ge \frac{1}{2}$ then $\alpha_1(\S)=0$.
\end{lemma}

\begin{proof}
We may assume $0\not\in\S$.
 Suppose $a \in \S$ with $\gcd(a, n)=1$. By the product-free property we have 
$a\S \cap \S = \emptyset.$ Now the gcd condition gives $|a\S| = |\S|$, whence
$ |\S| + |a\S| = 2|\S| \le n-1$
gives the result.
\end{proof}

This simple result already yields an upper bound for $D(n)$: one has, for all $n\ge8$, 
\begin{equation}
\label{eq-simple}
D(n)\le1-\frac1{3\log\log n}.
\end{equation}
To see this, if $\S$ is product-free (mod~$n$) and $D(\S)\ge\frac12$, then
the lemma shows that the set contains no $a$ with $(a, n)=1$,
whence $D(\S)\le1-\varphi(n)/n$.  The upper bound ~\eqref{eq-simple}
then follows from estimates of Rosser and Schoenfeld~\cite[Theorem 15]{RS62} valid for all
 $n \ge \e^{\e^2}.$
For $n$ with $8 \le n\le \e^{\e^2}$, we have from~\cite{PS11} that $D(n)<\frac12$,
which is stronger than~\eqref{eq-simple}.  
However, establishing the 
upper bound of Theorem~\ref{thm12}  is
more delicate.


\section{Linear Programs and Linear Programming Duality}\label{sec3}

In this section, for each fixed positive integer $n$, we formulate a linear program
$(P_n)$, along with its associated dual linear program $(D_n)$ which encodes
product-free conditions given in Section \ref{sec2};  related linear programs 
were already suggested in \cite[Question 3]{PS11} as an approach to upper bounds.
We
term $(P_n)$ a
primal linear program and $(D_n)$ its dual linear program, because
$(P_n)$ is given in a  standard inequality form called in the literature
{\em primal form} (alternatively, {\em canonical form}),
and $(D_n)$ takes the standard dual form 
as given  in  Schrijver \cite[eqn. (19), p. 91]{Sch86}, for example.

To label the variables in the  primal linear program $(P_n)$,  we let $u, v$ represent 
divisors of $n$ which are larger than $1$, and
we let  $\{u, v\}$ denote an unordered pair
of divisors with both $u, v >1$ and $uv\mid n$;  we permit the equality $u=v$
if $u^2 \mid n$.  The  linear program $(P_n)$ is as follows.


$$
\begin{array}{lcccl}
\mbox{\bf Primal LP}:~ { (P_n)}~~~~ &~&~  &\\
~&~&~&~&\\
 \mbox{MAXIMIZE}  & ~&\ell_P(\alpha) &= &
\sum_{u\mid n, u>1} \quad \frac{1}{u}\, \alpha_u\\
~~~~~subject ~~to~~
&~& ~&~&~\\
\mbox{nonnegativity constraints}: &~& \alpha_u & \ge & 0 \\
~~~~~and&~&~&~\\
\mbox{nontrivial constraints} ~C(\beta_u):  &~& \alpha_u &\le & 1\\
\mbox{nontrivial constraints}~C( \beta_{\{u, v\}}):& ~& \alpha_u + \alpha_v + \alpha_{uv} & \le & 2
\end{array}
$$
This linear program has $\delta_1(n)$ variables $\alpha_u$, where $\delta_1(n)$ denotes
the number of divisors of $n$ that exceed 1.
These are the variables which appear in the linear {\em objective function} 
$\ell_P(\alpha)$, where $\alpha$ denotes the vector of variables 
$\alpha= (\alpha_u)_{u\mid n, u>1}.$
We refer to the nonnegativity constraints as {\em trivial} constraints  and call all the 
other constraints  {\em nontrivial}.
The nontrivial constraints of this linear program are named after the variables
$\beta_u$ and $\beta_{\{u, v\}}$  that occur in the dual linear
 program $(D_n)$ described below. 
There are $\delta_1(n)+\delta_2(n)$ nontrivial constraints, where
 $\delta_2(n)$ counts the number of unordered pairs $\{u, v\}$ with $u,v>1$ and
$uv \mid n.$

 We let $L_{P}^{opt}(n)$ denote the optimal objective function 
of this linear program, which is the maximum possible value given the constraints, 
explicitly noting its dependence on $n$.
We note that 
Lemma~\ref{lem23} shows that the values of  $\alpha_u(\S)$ with $u>1$
for any product-free set $\S$ $(\bmod~n)$ give a feasible solution to $(P_n)$.

To a primal linear program $(P_n)$
 there is a canonically associated {\em dual linear program} $(D_n)$.
 To label the dual variables, 
 we let $u, v, w$ represent 
divisors of $n$ which are larger than $1$. Some dual variables are labeled by  unordered pairs
of divisors
e.g.\ $\{u, v\}$, and in this case we  require $uv\mid n$, and again we allow $u=v$ when
$u^2\mid n$.
The dual linear program  $(D_n)$ 
is  as follows.


$$
\begin{array}{lcccl}
\mbox{\bf Dual LP:} ~{ (D_n)}&~&~&~& ~ \\
~&~&~&~&\\
 \mbox{MINIMIZE}  &~&~~~~~ \ell_D(\beta)~~~\, =\, 
\sum_{u\mid n, u>1} \beta_u
+2\sum_{\{u,v\},\,uv \mid n, \, u, v > 1} \beta_{\{ u, v\}} &&\\
~~~~~subject ~~to & ~&~&~\\

\mbox{nonnegativity constraints}: &~& \beta_u & \ge & 0 \\
\mbox{nonnegativity constraints}: &~& \beta_{\{u, v\}} & \ge & 0 \\
~~~~~ and &~&~&~&~\\
\mbox{nontrivial constraints} ~C(\alpha_u):  &~& 
\beta_u + \sum_{\{v,w\},\,vw=u} \beta_{\{v, w\}} +
\sum^*_{v,\,uv\mid n} \beta_{\{u, v\}} & \ge&  \frac{1}{u}.
\end{array}
$$
The asterisk in $\sum^*$ signifies that the summand $\beta_{\{u,v\}}$ is
counted twice in the case that $v=u$.  (This corresponds to the primal LP 
constraint $C(\buv)$ taking the form $2 \alpha_{u} + \alpha_{uv} \leq
2$ when $u=v$.)

The nontrivial constraints $C(\alpha_u)$ in this linear program are named after  the
variables $\alpha_u$ in the primal linear program $(D_n)$; there are $\delta_1(n)$
of them.
The role of nontrivial constraints and variables
interchanges between the primal and dual linear programs; one sees
that $(D_n)$ has $\delta_1(n)+\delta_2(n)$ variables and $\delta_1(n)$ nontrivial 
constraints.  In addition
the  objective function coefficients and the constraint bound coefficients
interchange in the two programs. 
We let $L_{D}^{opt}(n)$ denote the optimal value
of the dual objective function $\ell_{D}(\beta)$,  
which is the minimal possible value given the constraints, 
explicitly noting its dependence on $n$.

Our results use only the following basic facts about LP duality.


\begin{prop}\label{pr31}
For each $n \ge 2$, the linear  programs $(P_n)$ and $(D_n)$ have equal optimal values:
$L_{P}^{opt}(n)= L_{D}^{opt}(n).$
 In particular,
 any  feasible solution $\beta= (\beta_u, \beta_{\{v, w\}}) $ of the dual linear program 
 $(D_n)$ has
 \begin{equation} \label{eq311}
 \ell_{D}(\beta) \ge L_{P}^{opt}(n).
 \end{equation}
\end{prop}

\begin{proof}
These are standard results in linear programming duality, 
see Schrijver \cite[Sec. 7.4, p. 90-91]{Sch86}.
The equality of primal and dual optimal  values holds whenever
both linear programs in a dual pair have a feasible solution (\cite[Corollary 7.1g, p. 90]{Sch86}).
Here these conditions are satisfied 
by inspection,  for $(P_n)$ we have  the feasible solution taking all $\alpha_u=0$,
and for $(D_n)$  we have the feasible solution taking all $\beta_u= \frac{1}{u}$
and all $\beta_{\{u,v\}}=0$. 

The inequality (\ref{eq311}) follows from {\em weak duality},
which asserts that any primal feasible solution $\alpha$
and dual feasible solution $\beta$ satisfy  $\ell_P(\alpha) \le \ell_{D} (\beta).$
Here this  is verifiable directly using the primal and dual constraints by
noting that
\begin{align*} 
\ell_{P}(\alpha) & = \sum_{u|n, u>1} \frac{1}{u} \alpha_u
 \le \sum_{u|n, u>1}\left(\beta_u + \sum_{\{v,w\},\,vw=u} \beta_{\{v, w\}} +
\sideset{}{^*}{\sum}_{v,\,uv\mid n} \beta_{\{u, v\}} \right) \alpha_u \\
& = \sum_{\substack{u\mid n\\u>1}} \beta_u \alpha_u + 
\sum_{\substack{\{v, w\}\\v, \,w>1, \,vw|n}} \beta_{\{v, w\}}( \alpha_v + \alpha_w + \alpha_{vw})\\
&\le \sum_{\substack{u\mid n\\u>1}}\beta_u + 
2 \sum_{\substack{\{v, w\}\\v,\, w>1, \,vw|n}} \beta_{\{v, w\}}
= \ell_{D}(\beta),
\end{align*}
as required. 
\end{proof}

We first note the following easy lower bound on the optimal
primal value $L_P^{opt}(n)$.\smallskip

\begin{prop}~\label{prop32}
For every $n \ge 2$ there holds
$$
L_P^{opt}(n) 
 \ge \frac{2}{3} \sum_{u|n, u>1} \frac{1}{u} .
$$
\end{prop}

\begin{proof}
We take all $\alpha_u =\frac{2}{3}$. 
This is obviously a feasible solution to  the linear program $(P_n)$
and its objective function value  
$\ell_{P}(\alpha) =\frac{2}{3}  \sum_{u|n, u>1} \frac{1}{u} .$
This value can be no larger than $L_P^{opt}(n)$, giving the result.
\end{proof}

For later use, we  restate the dual objective function in the  special case 
of a dual feasible solution that attains equality in all   the nontrivial constraints.
\begin{prop}
\label{restateddual}
In the dual linear program $(D_n)$ if a feasible solution $\beta$ attains  equality 
in all  the nontrivial
constraints $C(\alpha_u)$, then
$$
\ell_D(\beta)= \sum_{\substack{u\mid n\\ u>1}}\frac{1}{u}
-\sum_{\substack{\{v,w\}\\vw\mid n,~v,w>1}}\beta_{\{v,w\}}.
$$
\end{prop}
\begin{proof}
Assume that equality holds in all the nontrivial 
constraints of $(D_n)$. Adding them together yields
$$
\sum_{u|n,\, u >1} \beta_u + 3\sum_{\substack{\{v,w\}\\vw\mid n,\,v,w>1}}
 \beta_{\{v, w\}} =
 \sum_{u|n,\, u>1} \frac{1}{u}.
$$
(One checks here that each $\beta_{\{v, w\}}$ occurs exactly three times across all the constraints.)
Therefore, using the definition of $\ell_D(\beta)$, we have
$$
\ell_D(\beta)
=\sum_{u|n,\,u>1}\beta_u+2\sum_{\substack{\{v,w\}\\vw\mid n,\,v,w>1}}\beta_{\{v,w\}}
=  \sum_{u|n,\, u>1} \frac{1}{u} - \sum_{\substack{\{v,w\}\\vw\mid n,\,v,w>1}}
 \beta_{\{v, w\}},
$$
as asserted.
\end{proof}

\section{Primal Linear Program Bound}\label{sec4}

Our object is to  relate the 
bound for the primal  linear program $(P_n)$ to the density function $D(n)$.
We establish such a relation for integers of a special form.

Given a product-free set $\S$ $(\bmod~n)$,
note that $\S$ is the disjoint union of the sets $\S_u$ for $u\mid n$, so that
$
|\S|=\sum_{u\mid n}|\S_u|=
\sum_{u\mid n}|\T_u|\alpha_u=\sum_{u\mid n}\varphi\left(\frac nu\right)
\,\alpha_u,
$
and hence
$$
D(\S)=\frac1n|\S|= \sum_{u\mid n} \frac{1}{n} \varphi\left(\frac{n}{u}\right)\, \alpha_u.
$$
On the other hand, the linear program $(P_n)$
has the  objective function 
$$
\ell_{P}(\alpha) = \sum_{u|n, u>1} \frac{1}{u} \alpha_u.
$$
These two functions assign different weights to the variables $\alpha_u$.
These weights are related by the inequality
$$\frac{1}{n} \varphi\left(\frac{n}{u}\right)\ge \frac{\varphi(n)}{n} \, \frac{1}{u},$$
which  goes in the wrong direction for obtaining an upper bound,
but  has the positive feature that
  equality holds for those divisors $u$ of $n$  such that each prime factor 
of $u$ divides $n/u$.  The equality case gives exactly those $u$ such
that each  prime divisor of $u$ divides $n$ to a non-maximal power,
and in this case  the coefficient of these variables $\alpha_u$ in
$\ell_P(\alpha)$
is exactly $\frac{n}{\varphi(n)}$ times that of the same variable appearing in $D(\S)$.
 This suggests  that $D(n)$
be compared with $\frac{\varphi(n)}{n} L_{P}^{opt}(n)$,
and that  this be done in  cases  when all primes dividing
$n$ do so to a high power.  We obtain the following result,
which controls the loss from the inequality above.


\begin{theorem}\label{th41a}
Let $n$ be an arbitrary positive integer and set 
$$
X=X(n)= \lfloor \log n \rfloor,\quad
N=N(n)=\left(n\prod_{p\le X}p\right)^X.
$$
Then $n\mid N$ and
\begin{equation}\label{opt-bd}
D(N) \le \frac{\varphi(N)}{N} \left( 1 + L_{P}^{opt}(N)\right).
\end{equation}
\end{theorem} 

\begin{proof}
We first note that the theorem holds for all cases where $X=0$ or $1$, which correspond to
$n\le7$.
If $X=0$, then $N=1$ and $D(N)=0$, so the inequality holds.  If $X=1$, then $N=n$.
In each case up to $n=7$ we have $D(N)<\frac{1}{2}\le\varphi(N)/N$
except for $n=N=6$, in which case it is easy to see that
$D(N)=\frac13=\varphi(N)/N$.
Thus we may assume that $n\ge8$, and hence $X\ge2$.  

We next show that if $X \ge 2$ and $D(N) \le \frac{1}{2}$ then (\ref{opt-bd}) holds. This would follow
if we show that $\frac{\varphi(N)}{N} \left( 1 + L_{P}^{opt}(N)\right)> \frac{1}{2}$
holds when $X \ge 2$. We observe that
$$
\sum_{u\mid N}\frac1u\ge\prod_{p\mid N}\left(1+\frac1p+\frac1{p^2}+\dots+\frac1{p^X}\right)
\ge\prod_{p\mid N}\left(1+\frac1p+\frac1{p^2}\right).
$$
Using this fact together with Proposition \ref{prop32} and $X \ge 2$ we obtain
\begin{align*}
\frac{\varphi(N)}{N} \left( 1 + L_{P}^{opt}(N)\right) & \ge 
\frac{\varphi(N)}{N} \left(1 + \frac{2}{3} \sum_{u|N, u>1} \frac{1}{u}\right)
> \frac{\varphi(N)}N\cdot  \frac{2}{3}\sum_{u|N} \frac{1}{u} \\
&\ge\frac{2}{3}\prod_{p\mid N}\left(1-\frac1p\right)\left(1+\frac1p+\frac1{p^2}\right)
=\frac23\prod_{p\mid N}\left(1-\frac1{p^3}\right)>\frac2{3\zeta(3)}>\frac12.
\end{align*}

It  remains to treat the cases  where $X \ge 2$ and $D(N) > \frac{1}{2}$.  From~\cite{PS11}, this
implies that we may assume that $\omega(N)\ge6$.  Note that if $X\le5$, then
$n<\e^6<403$, so that there are at most two different primes greater than 5 dividing $n$,
and so $\omega(N)\le 5$.  Hence we may assume that $X\ge6$.

Now suppose  $\S$ is a product-free subset of $\Z/N\Z$ having $D(\S) > \frac{1}{2}$.
We take $\alpha_u := \alpha_u(\S)$, as in Lemma \ref{lem23}, whose values for $u|N, u>1$
give a feasible solution to $(P_N)$, and  Lemma \ref{lem22} gives $\alpha_1=0$.
Every $u\mid N$ is uniquely factorable as $u = bv$,
where $b\|N$ and $v\mid (N/b)/\rad(N/b)$.  We have
$\varphi(N/u)=\varphi(N)/(\varphi(b)v)$.  Thus,
$$
|\S|=
\sum_{\substack{u\mid N\\u>1}} |\S_{u}| = 
\sum_{\substack{u\mid N\\u>1}}\alpha_{u}\varphi\left(\frac {N}{u}\right)
=\varphi(N)\sum_{v\mid \frac N{\rad(N)}}\frac{\alpha_v}{v}
+\varphi(N)\sum_{\substack{b\|N\\ b>1}}\frac b{\varphi(b)}
\sum_{v\mid \frac{N/b}{\rad(N/b)}}\frac{\alpha_{vb}}{vb}.
$$
Using $\alpha_{vb}\le 1$, the second expression on the right is at most
$$
\varphi(N)  \sum_{\substack{b\|N\\ b>1}}
\frac1{\varphi(b)}\cdot\frac{\sigma(N/b)}{N/b}
\le\varphi(N) \sum_{\substack{b\|N\\ b>1}}
\frac1{\varphi(b)}\cdot\frac{N/b}{\varphi(N/b)}
=N \sum_{\substack{b\|N\\ b>1}}\frac1b ,
$$
using $\sigma(m)/m\le m/\varphi(m)$ (see~\cite[Theorem~329]{HW}).
Hence
\begin{equation}
\label{eq-binvolve}
\frac{1}{N}|{\S}|\le\frac{\varphi(N)}{N}\sum_{\substack{u\mid N\\u>1}}\frac{\alpha_u}{u}+
\sum_{\substack{b\|N\\ b>1}}\frac{1}{b}.
\end{equation}
We now claim that
\begin{equation}\label{bstack}
\sum_{\substack{b\| N\\ b>1}} \frac{1}{b} \, \le\, \frac{\varphi(N)}{N}.
\end{equation}
We defer its proof.
Using \eqref{eq-binvolve} and the claim (\ref{bstack}) 
we deduce that
$$
\frac1N|{\S}|  \le   
\frac{\varphi(N)}{N}\left(1+\sum_{u|N, u>1}\frac {\alpha_u}{u}\right) 
 \le  \frac{\varphi(N)}{N}\left( 1 + L_{P}^{opt}(N) \right).
$$
Since this holds for all product-free sets $\S\subset\Z/n\Z$  with $D(\S) > \frac{1}{2}$,
we conclude that the bound (\ref{opt-bd}) holds for $D(N)$, completing the
argument. 

It remains to prove the claim (\ref{bstack}).
Since each number $b$ with $b\|N$ is an $X$th power, we have
\begin{equation}
\label{eq-unitary}
\sum_{\substack{b\| N\\ b>1}}\frac{1}{b} <\sum_{m=2}^\infty\frac1{m^X}
<\frac1{2^X}+\int_2^\infty \frac{{\rm d}t}{t^X}\le\frac{1.4}{2^X},
\end{equation}
using $X\ge6$.
Since the number of distinct primes dividing $n$ that exceed $X$ is
at most $\log n/\log X<(X+1)/\log X<X$ for $X\ge6$, we have
\begin{equation}
\label{eq-premertens}
\frac{\varphi(N)}{N}= \prod_{\substack{p\mid n\\ p> X}}
\left(1- \frac{1}{p}\right)\prod_{p \le X} \left(1- \frac{1}{p}\right)
> \left(1-\frac{1}{X}\right)^{X} \prod_{p \le X} \left(1- \frac{1}{p}\right)
> \frac{1}{3} \prod_{p \le X} \left(1- \frac{1}{p}\right).
\end{equation}
Using an explicit estimate of Rosser and Schoenfeld
\cite[Corollary to Theorem 7]{RS62} and \eqref{eq-premertens}, 
we see that
\begin{equation}
\label{eq-mertens}
\frac{\varphi(N)}N>\frac1{3\e^\gamma\log X}\left(1-\frac1{\log^2X}\right)
\end{equation}
and so \eqref{eq-unitary} and \eqref{eq-mertens} imply that 
\eqref{bstack} holds when $X\ge6$.
\end{proof}


\section{Proof of Theorem~\ref{thm12}}\label{sec5} 

Let $n$ be a large integer, let $X=\lfloor\log n\rfloor$, and let
$$
N=N(n)=\left(n\prod_{p\le X}p\right)^X,
$$
as in Theorem~\ref{th41a}.
Lemma ~\ref{lem22} implies that any product-free set $\S$ having
$D(\S) \ge \frac{1}{2}$ necessarily has $\alpha_1=0$, and for these,
Lemma~\ref{lem23} shows that the remaining $\alpha_u$ with $u>1$
give a feasible solution to $(P_N)$.  

To bound the primal LP objective function from above, we 
investigate the 
dual linear program $(D_N)$.  A trivial choice for 
the variables $\beta$, in which all  the nontrivial constraints hold
with equality,  is to have each $\beta_u=1/u$ and each
$\beta_{\{u,v\}}=0$.  
This gives $L_{D}^{opt} (N) \le \sum_{u|N, u>1} \frac{1}{u}=\frac{\sigma(N)}{N}-1$.
Using   Theorem~\ref{th41a} and Proposition~ \ref{pr31}, we obtain
\begin{equation}
\label{eq-trivial}
D(N) \le \frac{\varphi(N)}{N}( 1+ L_{P}^{opt}(N))=
\frac{\varphi(N)}{N}( 1+ L_{D}^{opt}(N)) \le
\frac{\varphi(N)}{N}\frac{\sigma(N)}{N}<1,
\end{equation}
when $N>1$.
Using Theorem~\ref{th41a},~\eqref{eq-trivial} leads to an estimate of
the shape $D(n)<1-c/n^{\log 2}$,  which is much worse than our estimate~\eqref{eq-simple}.
However we  will  improve on this upper bound by deforming this solution via
``mass shifting'' from some of the variables
$\beta_u$ to the other variables $\beta_{\{v,w\}}$, while keeping all 
the nontrivial constraints tight.

To maximize the gain, Proposition \ref{restateddual} suggests that
one should move as much ``mass'' as possible onto the variables
$\beta_{\{u,v\}}$.
As a  critical parameter for the mass-shifting, we introduce 
\begin{equation}\label{eq-crit}
k=k(X)=\left\lfloor\frac{\rm e}4\log\log X\right\rfloor.
\end{equation}
We discuss this parameter choice in Remark \ref{rem-crit}
after the proof.

\begin{lemma}
\label{lem-estimates}
With the value of $k$ just defined, we have
$$
\binom{2k}{k}\asymp\frac{4^k}{\sqrt{k}}\asymp
\frac{(\log X)^{\frac{\rm e}2\log2}}
{\sqrt{\log\log X}}\asymp\frac{(\log\log X)^k}{k!}\asymp
\sum_{\substack{m\le X\\\Omega(m)=k}}\frac1m.
$$
\end{lemma}
\begin{proof}
The first three
 relations are clear from Stirling's formula and the
definition of $k$.  The last relation can be derived using  a famous theorem of Sathe
and Selberg \cite{Sel54} (see also~\cite[Theorem 7.19]{MV07}).
We use only the somewhat weaker version:
for all $x\ge20$ and $\epsilon>0$, over the range 
of integers $j$ with $1\le j\le(2-\epsilon)\log\log x$ the estimate
\begin{equation}
\label{eq-SS}
\sum_{\substack{m\le x\\\Omega(m)=j}}1\asymp\frac{x}{\log x}
\frac{(\log\log x)^{j-1}}{(j-1)!}
\end{equation}
holds uniformly, the implied constants depending only on $\epsilon$.
By partial summation, we have
\begin{align*}
\sum_{\substack{m\le X\\\Omega(m)=k}}\frac1m
&=
\frac1X\sum_{\substack{m\le X\\\Omega(m)=k}}1
+\int_{1}^{X}\frac1{t^2}
\sum_{\substack{m\le t\\\Omega(m)=k}}1\,{\rm d}t\\
&=\int_1^X\frac1{t^2}\sum_{\substack{m\le t\\\Omega(m)=k}}1\,{\rm d}t+O(1)
=
\int_{{\rm e}^{\sqrt{\log X}}}^{X}\frac1{t^2}
\sum_{\substack{m\le t\\\Omega(m)=k}}1\,{\rm d}t+O\left(\sqrt{\log X}\right).
\end{align*}
Using \eqref{eq-SS} and the already proved third relation, 
\begin{align*}
\int_{{\rm e}^{\sqrt{\log X}}}^X\frac1{t^2}
\sum_{\substack{m\le t\\\Omega(m)=k}}1\,{\rm d}t&\asymp
\int_{{\rm e}^{\sqrt{\log X}}}^X\frac1{t\log t}\frac{(\log\log t)^{k-1}}{(k-1)!}
\,{\rm d}t\\
&=\left(1-2^{-k}\right)
\frac{(\log\log X)^k}{k!}
\asymp
\frac{(\log X)^{\frac{\rm e}2\log2}}{\sqrt{\log\log X}}.
\end{align*}
Since $\frac{\rm e}2\log2>\frac12$, the error $O(\sqrt{\log x})$ is negligible, and
so the last relation in the lemma follows.
\end{proof}

Based on Lemma \ref{lem-estimates} we choose as a  weight parameter
$$A=A(X) :=c_0(\log X)^{\frac{\rm e}2\log2}/\sqrt{\log\log X}.$$
 where
$c_0$ is chosen large enough to assure that for all large $n$ we have both
\begin{equation}
\label{eq-recipsum}
A \ge \binom{2k}{k},\quad \frac{1}{2} A \ge \sum_{\substack{m\le X\\\Omega(m)=k}}\frac{1}{m}.
\end{equation}

We now define the variable values for a 
better feasible solution to the dual linear program $(D_N)$.
If $uv\mid N$, $u,v>1$, we set
$$
\beta_{\{u,v\}}:=\left\{ \begin{array}{ll} 
\frac1{uvA}, & ~\mbox{when}~u, v \le X\hbox{ and }\Omega(u)=\Omega(v)=k,\\
~&~\\
0 &~\mbox{otherwise}.
\end{array}
\right.
$$
We then choose the variables
$\beta_u$ by the rule
$$
\beta_u :=\frac{1}{u}-\sideset{}{^*}{\sum}_{\substack{v\\ uv\mid N}}
\beta_{\{u,v\}}-\sum_{\substack{\{v,w\}\\ vw=u}}\beta_{\{v,w\}},
$$
where we continue to understand that $u,v,w$ run over divisors of $N$
that exceed~1. That is, these variables are obtained from the $\beta_u$ in the ``trivial'' solution 
by subtracting off exactly the
amount required by the new $\beta_{\{u,v\}}$ needed to keep the
constraints $C(\alpha_u)$  tight. The parameter $A$ in the definition of $\beta_{\{u,v\}}$ 
serves as
a weight chosen (approximately) optimally  so that the new $\beta_u$ will remain nonnegative.

Thus, we have equality in the constraints $C(\alpha_u)$, and 
we next show  that we have nonnegativity for our variables $\beta_u$,
so that we have a dual feasible solution.
First note that if $\Omega(u)\ne k,2k$, then $\beta_u=1/u>0$.
Now suppose that $\Omega(u)=k$.  Then,
$$
\beta_u=\frac1u-\starsum_{\substack{v\\ uv\mid N}}\beta_{\{u,v\}}
\ge\frac1u-\frac2{uA}\sum_{\substack{v\le X\\ \Omega(v)=k}}\frac1v \ge 0,
$$
by \eqref{eq-recipsum}.  Finally suppose that $\Omega(u)=2k$.  Then,
$$
\beta_u=\frac1u-\frac1{uA}
\sum_{\substack{\{v,w\}\\vw =u\\ \Omega(v)=\Omega(w)=k}}1 \ge 0.
$$
The inequality holds
because the number of summands here is at most the number of partitions 
of a $2k$-element set into two $k$-element sets, which is 
$\frac{1}{2}\binom{2k}{k}< A$,
by \eqref{eq-recipsum}.

Thus, $\beta$ is feasible for $(D_N)$, and so $\ell_D(\beta)\ge L_P^{opt}(N)$,
by Proposition~\ref{pr31}.  We now get an upper bound for $\ell_D(\beta)$
using Proposition~\ref{restateddual}:
$$
\ell_D(\beta)=\sum_{\substack{u|N\\ u>1}}\frac1u-\sum_{\{u,v\}}\beta_{\{u,v\}}
\le  \sum_{\substack{u|N\\ u>1}}\frac1u-\frac{1}{2} 
\sum_{\substack{u,\,v\\ u,v\le X\\ \Omega(u)=\Omega(v)=k}} \beta_{\{u,v\}}
=\Big( \frac{\sigma(N)}{N}-1\Big) -\frac1{2A}\left(\sum_{\substack{u\le X\\\Omega(u)=k}}
\frac1u\right)^2.
$$
By Lemma~\ref{lem-estimates}, the sum here is of order 
$(\log X)^{\frac{\rm e}2\log2}/\sqrt{\log\log X}$, and $A$ is of
this order as well.  Since $\sigma(N)/N\le N/\varphi(N)$, we obtain
\begin{equation}\label{beta-bd}
 \ell_D(\beta)\le\Big(\frac{N}{\varphi(N)}-1\Big) -c_1\frac{(\log X)^{\frac{\rm e}2\log2}}
{\sqrt{\log\log X}}
\end{equation}
for some absolute constant $c_1>0$ and all sufficiently large $n$.

We next obtain an upper bound for $D(N)$. 
 Theorem ~\ref{th41a} 
and Proposition \ref{pr31}  combine
 with (\ref{beta-bd}) to yield
$$
D(N)   \le  \frac{\varphi(N)}{N}\Big( 1+ L_P^{opt}(N)\Big)
\le  \frac{\varphi(N)}{N}\Big( 1+ \ell_D(\beta)\Big)
 \le  \frac{\varphi(N)}{N} \left(\frac{N}{\varphi(N)} -
c_1\frac{(\log X)^{\frac{\rm e}2\log2}}{\sqrt{\log\log X}}\right).
$$
Now the lower bound (\ref{eq-mertens}) yields
$$
D(N) \le1-\frac{c}{(\log X)^{1-\frac{\rm e}2\log2}\sqrt{\log\log X}}
$$
for some positive constant $c$ and for all $n$ sufficiently
large.

Finally, since  $D(n)\le D(N)$ by Lemma~\ref{lem21},
this bound  applies to $D(n)$ as well.  But $\log X\le\log\log n$, so
$$
D(n)\le 1-\frac{c}{(\log\log n)^{1-\frac{\rm e}2\log2}\sqrt{\log\log\log n}}
$$
holds for $n$ sufficiently large. Since $D(n)<1$ for all $n$, by adjusting
$c$ if necessary, we have the inequality holding for all $n\ge20$.
This completes the proof of Theorem~\ref{thm12}.

\begin{remark}\label{rem-crit}
The choice of the critical parameter   (\ref{eq-crit}) in the argument is based on
specific features  of the dual LP.
Each  dual  variable  $\beta_{ \{u, v\}}$
appears  with weight $1$ in three nontrivial dual LP constraints,
namely in $C(\alpha_u)$,
$C(\alpha_v)$ and $C(\alpha_{uv})$. (If $u=v$ it appears in $C(\alpha_u)$
with weight $2$.)   If some mass is assigned to the variable $\beta_{\{u, v\}}$
this mass counts towards the constraint masses  $\frac{1}{u},
\frac{1}{v}, \frac{1}{uv}$ (i.e., the right hand
sides of the dual nontrivial constraints for which $\beta_{\{u, v\}}$
appears.) 
Now, for any fixed value of the parameter $k$, at least one of $w= u,
v, uv$ will satisfy either $\Omega(w) \le k$ or $\Omega(w) > 2k$.
This, together with the condition of equality of all dual constraints
(note that the contribution from the $\beta_u$-terms is positive), gives that
\begin{equation}\label{eq-key}
\sum_{ \{u,v\}} \beta_{\{u, v\}} \leq
\sum_{u : \Omega(u) \not \in [k+1,2k]} 
\left( 
\starsum_{v : uv |N} \beta_{\{u, v\}}
+ \sum_{ \{v,w\}: vw = u} \beta_{\{v, w\}}
\right)
\leq
\sum_{u : \Omega(u) \not \in [k+1,2k]}\,  \frac{1}{u},
\end{equation}
which imposes an upper bound on the total mass shifting.
The value (\ref{eq-crit}) for $k$ minimizes the right side, and establishes 
the strongest upper limit of this kind on
 the mass that can be moved.   Since a positive fraction of  the mass on
 the right side occurs  at level  $k$, this suggests 
 attempting to move  mass on exactly  this  level. The proof  then shows that this upper
 limit can be attained, up to a  constant factor.
Finally the restriction in the definition of $\beta_{\{u,v\}}$ to $u,v\le X$ is
convenient and does not appreciably alter the situation.
 \end{remark}

\section{Proof of Theorem \ref{thm11}}\label{sec6}

This result is proved by a modification of the proof of
Theorem~1 in \cite{KLP11}.

Let $\ell_x$ denote the least common multiple of the integers in
$[1,x]$, and let $n_x=\ell_x^2$.  As in the previous section,
let $k=k(x)=\lfloor\frac{\rm e}4\log\log x\rfloor$.  Instead of the
specific values $k$ and $2k$ which occurred in the previous section, the key
now is the interval $(k,2k)$.
 In the proof of Theorem~2.1 in~\cite{KLP11}, we showed that
$$
D(n_x)\ge
1-\frac{\pi(x)}{x}-
\frac{\varphi(n_x)}{n_x}\sum_{\substack{d\mid\ell_x\\ \Omega(d)\not\in(k,2k)}}
\frac{1}{d}
\ge1-\frac{\pi(x)}{x}-
\frac{\varphi(n_x)}{n_x} \sum_{\substack{P(d)\le x\\ \Omega(d)\not\in(k,2k)}}
\frac{1}{d},
$$
where $P(d)$ denotes the largest prime factor of $d>1$ (and $P(1)=1$).
Our result then followed from the bounds
$\varphi(n_x)/n_x\asymp1/\log x$ and $\log x\asymp\log\log n_x$,
 and from the estimate
 \begin{equation}\label{gooder}
 \sum_{\substack{P(d)\le x\\ \Omega(d)\not\in(k,2k)}} \frac{1}{d} 
 \ll (\log x)^{\frac{\rm e}2\log2}.
 \end{equation}

Our objective here is to improve on the estimate (\ref{gooder}) and
show that 
\begin{equation}
\label{eq-betterest}
\sum_{\substack{P(d)\le x\\\Omega(d)\not\in(k,2k)}}\frac1d
\ll\frac{(\log x)^{\frac{\rm e}2\log2}}{\sqrt{\log\log x}},
\end{equation}
from which Theorem ~\ref{thm11}  follows directly.
Towards doing  this, we prove the following lemma.
\begin{lemma}
\label{lem-recipsum}
Let $\epsilon>0$ be arbitrary but fixed.  
Let $\PP$ be a non-empty set of prime numbers and assume 
$s:=\sum_{p\in\PP}1/p<\infty$.
Let $\Nn_\PP$ denote the set of integers all of whose prime factors come from
$\PP$.  Then
$$
\frac{s^j}{j!}\le
\sum_{\substack{n\in\Nn_\PP\\\Omega(n)=j}}\frac1n\ll_\epsilon\frac{s^j}{j!}
$$
for every integer $0\le j\le(2-\epsilon)s$.
\end{lemma}

\begin{remark}
If the least prime in $\PP$ is $p_0$, then this  result can
be extended to the range $j\le(p_0-\epsilon)s$, with the implied constant then
depending on both $p_0$ and $\epsilon$.
\end{remark}

\begin{proof}
The lower bound is almost immediate and it holds for all $j$.  Indeed,
expanding $s^j$ by the multinomial theorem, each term is of the form
$b_n/n$ where $b_n$ is a multinomial coefficient, $n\in\Nn_\PP$, and
$\Omega(n)=j$.  Since $b_n\le j!$, the lower bound follows.  We note
that this argument also shows that $s^j/j!$ stands as an upper bound for the sum over
{\em squarefree} $n$ in the lemma.

For $n\in\Nn_\PP$, write $n=m^2u$, where $u$ is squarefree.
Since $\Omega(m^2)=2\Omega(m)$,
we have by the observation above about squarefree numbers,
$$
W_j:=\sum_{\substack{n\in\Nn_\PP\\\Omega(n)=j}}\frac1n\le
\sum_{\substack{m\in\Nn_\PP\\\Omega(m)\le j/2}}\frac1{m^2}\cdot
\frac{s^{j-2\Omega(m)}}{(j-2\Omega(m))!}.
$$
This, together with $j!/(j - 2 \Omega(m))! \leq j^{2 \Omega(m)}$,
gives that
\begin{align*}
W_j&\le\frac{s^j}{j!}\sum_{\substack{m\in\Nn_\PP\\\Omega(m)\le j/2}}\frac1{m^2}
\cdot\frac{j!}{(j-2\Omega(m))!}s^{-2\Omega(m)}
\le \frac{s^j}{j!}\sum_{\substack{m\in\Nn_\PP\\\Omega(m)\le j/2}}\frac1{m^2}
\cdot\left(\frac{j}{s}\right)^{2\Omega(m)}\\
&\le \frac{s^j}{j!}\sum_{\substack{m\in\Nn_\PP\\\Omega(m)\le j/2}}
\frac{(2-\epsilon)^{2\Omega(m)}}{m^2}
\leq\frac{s^j}{j!}\prod_{p\in\PP}
\left(1+  \sum_{i=1}^{\infty} \frac{(2-\epsilon)^{2i}}{p^{2i}}\right)\\
&=\frac{s^j}{j!}\prod_{p\in\PP}
\left(1-\left(\frac{2-\epsilon}{p}\right)^{2}\right)^{-1}
\ll_\epsilon\frac{s^j}{j!},
\end{align*}
and the proof of the lemma is complete.
\end{proof}

We now prove \eqref{eq-betterest}, which as we have seen, is sufficient
for the proof of Theorem~\ref{thm11}.  We may assume that
$x$ is large.  Let $\PP$ be the set of
primes in $[1,x]$, so that $\Nn_\PP$ consists of the integers $n$ with
$P(n)\le x$ and $s=\log\log x+O(1)$.  Let $a_j=s^j/j!$
and note that if $j<k$, then $a_j/a_{j+1}$ is bounded below 1.  Thus,
by Lemma~\ref{lem-recipsum} and Lemma~\ref{lem-estimates}, we have
$$
\sum_{\substack{P(n)\le x\\\Omega(n)\le k}}\frac1n
\ll\sum_{j\le k}a_j
\ll a_k\ll \frac{(\log x)^{\frac{\rm e}2\log2}}{\sqrt{\log\log x}}.
$$
Similarly, $a_{j+1}/a_j$ is bounded below 1 when $j\ge2k$, so that by
Lemma~\ref{lem-recipsum} applied for $2k\le j\le 2.5k$ ($<1.7\log\log x$), 
$$
\sum_{\substack{P(n)\le x\\2k\le\Omega(n)\le2.5k}}\frac1n
\ll\sum_{2k\le j\le2.5k}a_j\ll a_{2k}
\ll \frac{(\log x)^{\frac{\rm e}2\log2}}{\sqrt{\log\log x}},
$$
the last inequality holding as in the third relation in Lemma~\ref{lem-estimates}.
It remains to consider those $n$ with $\Omega(n)>2.5k$.  
Using~\cite[Corollary 2.5]{KLP11}, we have that
$$
\sum_{\substack{P(n)\le x\\ \Omega(n)\ge(2.5\e/4)\log\log x}}\frac1n
\ll(\log x)^{-(2.5\e/4)\log(2.5/4)}<(\log x)^{0.8},
$$
which is negligible.
This then proves~\eqref{eq-betterest} and Theorem~\ref{thm11}.

\section*{Acknowledgments}
Part of this work was done while the three authors visited MSRI,
as part of the semester program ``Arithmetic Statistics.''
They thank MSRI for support, funded through the NSF.
The first author was supported in part by grants from
the G\"oran Gustafsson Foundation, and the Swedish Research Council.
The second author was supported in
part by NSF grants DMS-0801029 and DMS-1101373.  The third author was supported in
part by NSF grant DMS-1001180.


\end{document}